\documentclass[12pt,a4paper]{article}
\usepackage{amssymb}

\usepackage{graphicx}
\usepackage{amsmath}

\usepackage{amsfonts}
\usepackage{amsthm,amscd}
\usepackage{graphicx}
\usepackage{amsmath}
\usepackage{amssymb}
\usepackage{amstext}

\newtheorem{theorem}{Theorem}

\newtheorem{corollary}[theorem]{Corollary}

\newtheorem{example}[theorem]{Example}
\newtheorem{lemma}[theorem]{Lemma}

\begin{document}

\title{Selection of calibrated subaction when temperature goes to zero in the discounted problem}

%: positive and zero temperature,
%selection and non-selection}
\author{Renato Iturriaga, Artur O. Lopes and Jairo K. Mengue}

\date{\today}
\maketitle

\abstract{

Consider $T(x)= d \,  x$ (mod 1) acting on $S^1$, a Lipschitz potential $A:S^1 \to \mathbb{R}$, $0<\lambda<1$
and the unique function
$b_\lambda:S^1 \to \mathbb{R}$ satisfying $
 b_\lambda(x) = \max_{T(y)=x} \{ \lambda \, b_\lambda(y) + A(y)\}.
$

%It is known that any convergent subsequence $b_{\lambda_n}- \sup b_{\lambda_n} $, $\lambda_n\to 1$,
%determines in the limit a calibrated subaction for $A$.

We will show that, when $\lambda \to 1$, the function $b_\lambda- \frac{m(A)}{1-\lambda}$ converges uniformly to the calibrated subaction $V(x) = \max_{\mu \in   \mathcal{ M}}     \int S(y,x) \, d \mu(y)$, where $S$ is the Ma\~ne potential, $\mathcal{ M}$ is the set of invariant probabilities with support on the Aubry set and $m(A)= \sup_{\mu \in \mathcal{M}} \int A\,d\mu$.

For $\beta>0$  and $\lambda \in (0,1)$, there exists a unique fixed point $u_{\lambda,\beta} :S^1\to \mathbb{R}$ for the equation $e^{u_{\lambda,\beta}(x)} = \sum_{T(y)=x}e^{\beta A(y) +\lambda u_{\lambda,\beta}(y)}$.
It is known that as $\lambda \to 1$ the family  $e^{[u_{\lambda,\beta}- \sup u_{\lambda,\beta}]}$ converges uniformly to the main eigenfuntion $\phi_\beta $ for the Ruelle operator associated to $\beta A$.
We consider $\lambda=\lambda(\beta)$, $\beta(1-\lambda(\beta))\to+\infty$ and $\lambda(\beta) \to 1$, as $\beta \to\infty$. Under these hypotheses    we will show that   $\frac{1}{\beta}(u_{\lambda,\beta}-\frac{P(\beta A)}{1-\lambda})$ converges uniformly to the above $V$, as $\beta\to \infty$. The parameter
$\beta$ represents  the inverse of temperature in Statistical Mechanics and $\beta \to \infty$ means  that we are considering that the temperature goes to zero. Under these conditions we get selection of subaction when $\beta \to \infty$.

}

\section{Introduction}

Consider  $T(x)= d \,  x$ (mod 1) acting on $S^1$ and a Lipschitz potential $A:S^1 \to \mathbb{R}$.
We denote by
\begin{equation}\label{eq1}
m(A)= \sup \left\{\int A \,d \mu, \,\text{where}\, \mu\, \text{is invariant for}\, T \right\}.
\end{equation}
Any invariant $\mu$ which attains this supremum is called an {\bf $A$-maximizing probability}.

A {\bf subaction} for $A$ is a continuous function $D: S^1 \to \mathbb{R} $  such that for all $x\in S^1$,
$$D(T(x)) \geq A(x) + D(x) - m(A).$$
It is called a {\bf calibrated subaction} if for all $y\in S^1$, $$ D(y) = \max_{T(x)=y} \{ \,  A(x)+D(x) - m(A)\}.$$

We refer the reader to \cite{BLL, Ga, LoT, PP}
for general results on Ergodic Optimization and Thermodynamic Formalism.

Maximizing probabilities and calibrated subactions are dual objects in Ergodic Optimization. On the one hand $m(A)$ satisfies (\ref{eq1}), but on the other hand
\begin{align*}
m(A)&= \inf_{H \, \text{continuous}} \left( \,\sup_{x\in S^1} (A(x)+H(x)-H(T(x)) )\,\right)\\
&= \sup_{x\in S^1} (A(x)+D(x)-D(T(x))
\end{align*}
for any calibrated subaction $D$. Furthermore, it is known that a calibrated subaction can help to identify the support of the maximizing probabilities for $A$ (see \cite{CLT} or \cite{BLL}).

 A natural problem is: how to find subactions?  Note that we need to have at hand the exact value $m(A)$ in order to verify if a specific candidate $D$ is indeed a calibrated subaction.
The discounted method, which is described below, can be quite useful in order to get a good approximation (via iteration of a contraction) of a calibrated subaction without the knowledge of the value $m(A)$.

For each fixed  $\lambda \in (0,1)$, consider the function
$b=b_\lambda:S^1 \to \mathbb{R}$ satisfying the equation
\begin{equation}
b(x) = \max_{T(y)=x} \{ \lambda \, b(y) + A(y)\}. \label{blambda}
\end{equation}
This function is unique and we call  $b_\lambda$ the {\bf$\lambda$-calibrated subaction} for $A$ (see for instance Theorem 1 in \cite{Bou} or  \cite{LO}).
%Clearly it satisfies, $b_\lambda (T(z)) - \lambda b_\lambda(z) \geq A(z)$ for any $z\in S^1$.

The solution $b_\lambda$ can be obtained in the following way: consider $\tau_j, \,j=1,...,d$  the inverse branches of $T$.
For $\lambda<1$, consider
$$ S_{\lambda,A} (x,a) = \sum_{k=0}^\infty \lambda^k A(\,(\tau_{a_k}\circ \tau_{a_{k-1}}\circ\,...\, \circ\tau_{a_0})\, (x)  \,) ,$$
where $a\in \{1,2,...,d\}^\mathbb{N}$ (see expression (6) in \cite{LO} or \cite{T1}). Then,  the solution of (\ref{blambda}) is given by
$b_\lambda(x)=\sup_{a \in \{1,2,...,d\}^\mathbb{N}}S_\lambda (x,a)$ (see section 3 in \cite{LO}.). For fixed $x$, as the function $S_{\lambda,A} (x,.)$ is continuous in the second coordinate, there exists some $a$ {\bf realizing the supremum}, which is  called a $b_\lambda (x)$ {\bf realizer}. Denoting $a=a_{0}a_{1}...$ we have that for any $k$:
\begin{equation}\label{eq3}
  b_{\lambda}(\tau_{a_k}\circ...\circ \tau_{a_0}x) = \lambda b_{\lambda}(\tau_{a_{k+1}}\circ...\circ \tau_{a_0}x) + A(\tau_{a_{k+1}}\circ...\circ \tau_{a_0}x)
\end{equation}
and $$
b_{\lambda}(x) = \lambda^{k+1}b_{\lambda}(\tau_{a_{k}} \circ...\circ \tau_{a_0}\,x) +A(\tau_{a_0}\,x)+...+\lambda^{k}A(\tau_{a_{k}} \circ...\circ \tau_{a_0}\,x),
$$
while for a general $a \in \{1,2,...,d\}^\mathbb{N}$ we have that
$$
b_{\lambda}(x) \geq \lambda^{k+1}b_{\lambda}(\tau_{a_{k}} \circ...\circ \tau_{a_0}\,x) +A(\tau_{a_0}\,x)+...+\lambda^{k}A(\tau_{a_{k}} \circ...\circ \tau_{a_0}\,x).
$$

The above family $b_\lambda$, $0<\lambda<1$, is equicontinuous. It is known that any convergent subsequence $b_{\lambda_n}- \sup b_{\lambda_n} $, $\lambda_n\to 1$,
determines in the limit a calibrated subaction for $A$ (see Theorem 1 in \cite{Bou}, \cite{BKRU} or \cite{T1}). This procedure, known as the discounted method, is borrowed from
Control Theory. The function $b_{\lambda}$ is obtained as a fixed point of a contraction (see \cite{LO}), which is, of course, computationally convenient (note that  $m(A)$ does not appear on expression (\ref{blambda})).

\medskip

In this work the constant $\sup b_{\lambda}$ will be replaced by $\frac{m(A)}{1-\lambda}$. Our first aim is to describe a certain calibrated subaction as the limit of $b_{\lambda}- \frac{m(A)}{1-\lambda}$, as $\lambda\to 1$. In this way the limit will not depend of the convergent subsequence. Later we will consider eigenfunctions of the Ruelle operator and selection of subaction via the limit when temperature goes to zero (see \cite{CLT} and \cite{BLL}). % We point out that the {\bf analogous results are true on the symbolic space $\{1,2,..,d\}^\mathbb{N}.$}

\medskip
A point $x$  in $S^1$ is called {\bf $A$-non-wandering}, if for any $\epsilon>0$, there exists $z\in S^1$ and $n\geq 1$,  such that,
 $d(z,x)<\epsilon, T^n (z)=x$, and $|\sum_{j=0}^{n-1} (A - m(A))\, (T^j (z)) |<\epsilon$.
 The {\bf Aubry set} for $A$ is the set of $A$-non-wandering points in $S^1$ and is denoted by $\Omega(A)$. Any invariant probability with  support inside the Aubry set is $A$-maximizing (see \cite{CLT} or section 3 in \cite{CLO}). We denote by $\mathcal{M}$ the set of $A$-maximizing probabilities.

Consider the function  $S$ given by Definition 22 in \cite{CLT} (see also \cite{GL2} and \cite{Ga}):
$$S(y,x)=  \lim_{\epsilon \to 0}
\sup\{S_n(A - m(A)) (y')\,|\, n\geq 1,\,d(y',y)\leq \epsilon,\ \, T^n (y')=x\,\},
$$
where
\begin{equation} \label{forg}S_n(A - m(A)) (y')=\sum_{j=0}^{n-1}(A-m(A))(T^j(y')).
\end{equation}
This function is called the {\bf Ma\~ne potential}.
\medskip

{\bf Remark 1:} For $y\in \Omega(A)$ fixed, the function $S(y,.)$ is a calibrated subaction (see \cite{CLT}, Proposition 5.2 in \cite{Ga} or \cite{GL2}) with the same Lipschitz constant as $A$. As $S(y,y)=0$ for $y \in \Omega(A)$, we get that for fixed $y \in \Omega(A)$, the function $x \to S(y,x)$ is bounded.

\medskip

The function $S(x,y)$ is not continuous as a function of $(x,y)$ (see Proposition 3.5 in \cite{CLT1})

We define, for each $x \in S^1$, the  subaction
$$V(x) = \max_{\mu \in   \mathcal{ M}}     \int S(y,x) \, d \mu(y)   .$$

Our first goal is to prove the following theorem.

\begin{theorem}\label{goal} If $A$ is Lipchitz, when $\lambda \to 1$, the function $U_{\lambda}:=\left(b_{\lambda}-\frac{m(A)}{1-\lambda}\right)$ converges uniformly to $V$.
\end{theorem}

The above theorem assures that the  subaction $V$ is very special among the possible ones. We will prove that $V$ is calibrated (see corollary \ref{Vcalibrated}).

\bigskip

This first part of the paper follows the ideas in \cite{DFIZ} and \cite{DFIZ1} and  obtains analogous results  in the framework of ergodic optimization.

\medskip

From now on we describe the second part of our paper which considers a limit when  temperature goes to zero - this result has a quite  different nature when compared with  \cite{DFIZ} and \cite{DFIZ1}.

It is known (see for instance section 7 in \cite{Bou}, \cite{LoT} or \cite{LMMS}) that, for fixed $\beta>0$ and $\lambda \in (0,1)$, there exists a unique fixed point $u_{\lambda,\beta} :S^1\to \mathbb{R}$ for the equation
\[e^{u_{\lambda,\beta}(x)} = \sum_{T(y)=x}e^{\beta A(y) +\lambda u_{\lambda,\beta}(y)}.\]
For fixed $\beta$, the family $u_{\lambda,\beta}$ is equicontinuous in $\lambda \in(0,1)$.
Moreover, $e^{u_{\lambda,\beta}- \sup u_{\lambda,\beta}}$  converges as $\lambda \to 1$. The limit function  $\phi_\beta $ is an eigenfunction associated to the main eigenvalue (spectral radius) of the Ruelle operator $L_{\beta A}$ associated to the potential $\beta A$ (see Lemma \ref{caselambda}).

The parameter
$\beta$ represents  the inverse of temperature in Statistical Mechanics

From \cite{CLT}  we get that $\frac{1}{\beta} \log \phi_\beta$ is equicontinuous and any limit of a convergent subsequence of
$\frac{1}{\beta_n} \log \phi_{\beta_n}$, $\beta_n \to \infty$, is a calibrated subaction. The  limit when $\beta \to \infty$ is known as the limit when  temperature goes to zero (see \cite{BLL}).  For the case of locally constant potentials a precise description is \cite{Bre}. The most comprehensive result of this kind of convergence for Lipschitz potentials is Theorem 16 in \cite{GT}.

In the {\bf standard} terminology we say that there exists selection of subaction at zero temperature if the limit of the function $\frac{1}{\beta} \log \phi_\beta$ exists, when $\beta \to \infty$ (see \cite{BLL} for general results).

For a fixed value $\beta$, the function $u_{\lambda,\beta}$ is obtained as a fixed point of a contraction. In this way, one can get a computable method (via iteration of the contraction) for getting good approximations of the main eigenfunction $\phi_\beta$ (taking $\lambda$ close to $1$).

We denote by $\alpha(\beta)$ the main eigenvalue which is associated to $\phi_\beta$. The pressure $P(\beta A)=\log (\alpha(\beta))$ is equal to
$$ \sup_{\mu\,\,\text{invariant for} \,\, T}\,\{ \int h(\mu)+ \beta\, \int A d\, \mu\},$$
where $h(\mu)$ is the Kolmogorov entropy of $\mu$ (see \cite{PP}).

%We will need to take the value $0<\lambda<1$ as a function of $\beta$, when $\beta \to \infty$,  in order to get convergence in the next theorem.

Our second goal is to show:
\begin{theorem}  \label{goal2}
	 Suppose $A$ is Lipchitz and there exists a relation between $\beta$ and $\lambda$ so that $\beta(1-\lambda(\beta))\to+\infty$ and $\lambda(\beta) \to 1$, as $\beta \to\infty$. Then, when $\beta\to \infty$, $\frac{1}{\beta}(u_{\lambda,\beta}-\frac{P(\beta A)}{1-\lambda})$ converges uniformly to $V$.
\end{theorem}

The above result requires a control of the velocity such that $(1-\lambda(\beta))$ goes to zero, vis-\`a-vis, the growth of $\beta$ to infinity. We will show on the end of the paper a counterexample  proving that $\frac{1}{\beta} \log \phi_{\beta}$ does not always converge to $V$. This shows  that is really necessary some kind of joint control of the parameters $\beta$ and $\lambda$ (as considered in the above theorem).

The last theorem shows the existence of  selection of subaction when temperature goes to zero (but in a sense which is not the standard form).

In the next section we will prove  Theorem \ref{goal} and in the last section we will prove  Theorem \ref{goal2}.

%As we said before
The results that we get here are also true if one considers $T(x)=\sigma(x)$ acting on Bernoulli space $\{1,...,d\}^{\mathbb{N}}$.

\section{The limit of the $\lambda$-calibrated subaction, when $\lambda \to 1$.}

We start with the following result:

\begin{lemma}\label{lemmai} $V$ is a  subaction for $A$  and also satisfies: \newline
1. $\int V d \mu \geq 0$, for any $\mu \in  \mathcal{ M}.$ \newline
2. If $w$ is  a calibrated subaction, such that, $\int w \,d \mu\,\geq 0$, for any $\mu \in \mathcal{ M}$, then $w\geq V.$ %\newline
%
%3. for any $x$ we get that $V(x) = \inf_{w \in \mathcal{F}} \,\{w(x)\} .$

\end{lemma}
\begin{proof}

From Remark 1 we get that, for each $y \in \Omega(A)$, the function $S(y,\cdot)$ is Lipschitz continuous and the Lipschitz constant does not depend of $y$. Therefore $V$ is Lipschitz. For any $x,y\in S^1$  we have $S(y,T(x)) \geq A(x) -m(A)+ S(y,x)$. Then, for any $\mu \in \cal{M}$ we get
$\int S(y,T(x)) \, d\mu(y) \geq A(x) -m(A)+ \int S(y,x)\, d\mu(y) $   which proves that $V(T(x))\geq A(x) + V(x) -m(A)$.

\bigskip

Proof of 1.: It is known (Prop. 23 page 1395 \cite{CLT}) that $S(y,x) + S(x,y)\leq 0$.
We say that  $x$ and $y$  in $\Omega(A)$ are in the same static class if $S(y,x) + S(x,y)= 0$.

For an ergodic maximizing measure $\mu$ we claim that  for $\mu$-almost every pair $x$ and $y$ these two points are in the same static class.

First note that as $A$ is Lipschitz  the function $S$ can be also written as
$$S(y,x)=$$
$$  \lim_{\epsilon \to 0}
\sup\{S_n(A - m(A)) (y')\,|\, n\geq 1,\,d(y',y)\leq \epsilon,\, d(T^n (y'),x)\leq \epsilon\}.
$$

Now we will show the claim: given an ergodic probability $\mu$,  an integrable function $F$ and a Borel set $B$, such that
$\mu(B)>0$, denote by  $\tilde{B}$  the set of points $p$ in  $B$, such that, for all $\epsilon$, there exists
an integer $N > 0 $, with the properties:\[ T^N(p)\in  B \,\,\,\text{  and}\,\,\,\,
 |\sum_{j=0}^{N-1} F(T^j) (p) - N \int F \, d \mu|<\epsilon.\]
It is known that $\mu(B)=\mu(\tilde{B})$ (see \cite{Mane}). This results is known as Atkinson's Theorem.

Let $x, y \in supp(\mu)$ where $x$ is a  Birkhoff point for $A$ and fix $\rho>0$. As $\mu(B(y,\rho))>0$, for some $n$ we have $T^n(x)\in B(y,\rho)$. Let $\rho'<\rho$ be such that for any $x' \in B(x,\rho')$ we have $d(T^j(x),T^j(x'))< \rho$, for all $j\in\{0,...,n\}$.  Consider the set $B=B(x,\rho')$. There exists $p\in B(x,\rho')$, such that, for all $\epsilon$, there exists
an integer $N > 0 $, satisfying $T^N(p)\in  B(x,\rho')$  and $|\sum_{j=0}^{N-1} A(T^j) (p) - N \int A \, d \mu|<\epsilon$. As this is true for a. e. $p\in B(x, \rho')$ we can suppose that $N>n$, replacing $p$ by $T^N(p)$ in the case it is necessary.
It follows that $S(x,y)+S(y,x)=0$ which proves the claim.

As $ S(x,y) = - S(y,x) $ for any $ x,y \in supp(\mu)$, it follows from item 1 of Proposition 3.1 in \cite{CLT1} that there exists $Q>0$, such that,
$-Q < S(x,y) < Q $, for any $x,y \in supp(\mu)$.

We have $\int \int  S(x,y)+S(y,x) \, d \mu(x) \, d\mu (y) =0$, for any $\mu \in \mathcal{ M}$ which is ergodic, then
$\int \int  S(x,y) \, d \mu(x)  d\mu (y)=0,$ using Fubini's Theorem ($S(x,y)$ is integrable by last paragraph).
It follows that
$$ \int V(x) d \mu(x) \geq \int \,\,[ \int S(y,x) \, d \mu(y)  ] \,\,\, d\mu(x)= 0,$$
for any ergodic probability $\mu \in \mathcal{ M}$. The same inequality for a general $\mu \in \mathcal{M}$ follows from the ergodic decomposition theorem.

\bigskip

Proof of 2.: We know (see Theorem 10 in \cite{GL2}) that, for any calibrated subaction $w$ and any $x$,
$$ w(x)= \sup_{ y \in \Omega(A)} \{ \, w(y) + S(y,x)\}.$$
Therefore, for $\mu \in \mathcal{ M}$, such that, $\int w(y) \, d \mu(y)\geq 0,$ we get
$$ w(x)\geq w(x) - \int w(y) \, d \mu(y)\geq \int S(y,x) \, d \mu(y) .$$
If $\int w(y) \, d \mu(y)\geq 0$ for any $\mu \in \cal{M}$ we obtain $w\geq V$.

%\bigskip

%Proof of 3.: This follows at once  from 1. and 2.

\end{proof}

%We observe that Item 3. above corresponds to Proposition 1.2 (1) in \cite{DFIZ1}.

\bigskip

Given $\lambda$, $y$ and a $b_\lambda(y)$ realizer $a=a_{0}a_{1}...$, consider the probability
\begin{equation}
\mu_\lambda^y= (1-\lambda) \,\sum_{n=0}^{\infty}\, \lambda^n  \,\delta_{(\, \tau_{a_{n}}\,\circ\,...\,\circ\, \tau_{a_{0}}\, ) (y)}. \label{prob}
\end{equation}  From (6) in \cite{LO} we get that $b_\lambda(y) = S_{\lambda,A}(y,a(y))$, where $a(y)$ is a realizer of $y$, then for any $y$ we have that $$b_\lambda(y) = \frac{1}{(1-\lambda)} \int A \,\,d\mu_\lambda^y.$$
We will show that any limit probability of $\mu_\lambda^y$, as $\lambda \to 1$, belongs to $\cal{M}$.

\begin{lemma}  For any continuous function $w:S^1\to\mathbb{R}$, and probability $\mu_\lambda^y$ as above,  we get
		\begin{equation} \label{io}
	\int w\circ T\,d\mu_{\lambda}^{y} - \int w \,d\mu_{\lambda}^{y}(x)= (1-\lambda) (\,w(y)-   \int w \,d\mu_{\lambda}^{y}\,). \end{equation}
\end{lemma}

\begin{proof} Indeed,
$$\int (\,w\circ T -  w\,)\,d\mu_{\lambda}^{y} $$
$$ = (1-\lambda) [\, \sum_{n=0}^\infty \lambda^n \,w(T( (\, \tau_{a_{n}}\,\circ\,...\,\circ\, \tau_{a_{0}}\, ) (y) )) -  \sum_{n=0}^\infty \lambda^n \,w((\, \tau_{a_{n}}\,\circ\,...\,\circ\, \tau_{a_{0}}\, ) (y) )\,]  $$
$$ = (1-\lambda) [\,w(y) + \sum_{n=1}^\infty \lambda^n \,w( (\, \tau_{a_{n-1}}\,\circ\,...\,\circ\, \tau_{a_{0}}\, ) (y) ) -  \sum_{n=0}^\infty \lambda^{n} \,w((\, \tau_{a_{n}}\,\circ\,...\,\circ\, \tau_{a_{0}}\, ) (y) )\,]  $$
$$ = (1-\lambda) [\,w(y) + \sum_{n=0}^\infty \lambda^{n+1} \,w( (\, \tau_{a_{n}}\,\circ\,...\,\circ\, \tau_{a_{0}}\, ) (y) ) -  \sum_{n=0}^\infty \lambda^{n} \,w((\, \tau_{a_{n}}\,\circ\,...\,\circ\, \tau_{a_{0}}\, ) (y) )\,]  $$
$$ = (1-\lambda) [\,w(y) + (\lambda-1) \sum_{n=0}^\infty \lambda^n \,w( (\, \tau_{a_{n}}\,\circ\,...\,\circ\, \tau_{a_{0}}\, ) (y) )\,]  $$
$$ = (1-\lambda) [\,w(y) -  \int w\,  d\mu_{\lambda}^{y}\,].$$

\end{proof}

\begin{lemma}  \label{se} Given $y\in S^1$,  any accumulation probability $\mu_{\infty}$, in the weak* topology, of a convergent subsequence $\mu_{\lambda_i}^{y}$, $\lambda_i \to 1$, belongs to $\cal{M}$.
\end{lemma}

\begin{proof}
It follows from above lemma that $\mu_{\infty}$ is invariant.
Moreover, by (\ref{eq3}) and definition of $\mu_{\lambda}^{y} $ we have that
\[\int b_{\lambda_{i}}(T(x)) - \lambda_{i}b_{\lambda_{i}}(x) - A(x) \, d \mu_{\lambda_i}^{y}(x) = 0.\]
Then,
\[\int b_{\lambda_{i}}(T(x))-b_{\lambda_{i}}(x) \, d \mu_{\lambda_i}^{y}(x) +\int   (1- \lambda_{i})b_{\lambda_{i}}(x) - A(x) \, d  \mu_{\lambda_i}^{y}(x)= 0.\]
When, $i\to\infty$ the left integral converges to zero. Therefore,
\[\lim_{i\to\infty} \int   (1- \lambda_{i})b_{\lambda_{i}}(x) - A(x) \, d  \mu_{\lambda_i}^{y}(x) = 0.\]
It is known (see for instance end of Theorem 11 in \cite{BCLMS} or \cite{LoT}) that $(1- \lambda)\inf b_{\lambda} \to m(A)$, uniformly with $\lambda \to 1$. Therefore,
$$\int ( m(A) - A )\, d\mu_{\infty} = \lim_{i\to\infty} \int   [\,(1- \lambda_{i})\inf b_{\lambda_{i}} - A(x) \,]\, d  \mu_{\lambda_i}^{y}(x)\leq$$
$$ \lim_{i\to\infty} \int  [\, (1- \lambda_{i})b_{\lambda_{i}}(x) - A(x) \,]\, d  \mu_{\lambda_i}^{y}(x) = 0,$$
proving the claim.

\end{proof}

\begin{lemma}\label{dual}
The family of functions \[U_{\lambda}:= b_{\lambda} - \frac{m(A)}{1-\lambda}\]
is equicontinuous and uniformly bounded. Furthermore, for any maximizing probability $\mu \in \cal{M}$ we have
\[\int U_{\lambda} \, d\mu \geq 0,\,\,\,\,\,\forall \lambda \in (0,1),\]
and for any  subaction $w$ we have
\[ U_\lambda (y) \leq  \,w(y) \,- \int w \,d \mu_{\lambda}^{y},\,\,\, \,\, \forall \lambda \in (0,1)\,,\,\forall y\in S^1.\]
\end{lemma}
\begin{proof}
As  $b_{\lambda}(T(z)) - \lambda b_{\lambda}(z) - A(z)\,\geq 0,$  for any maximizing probability $\mu \in \cal{M}$ we have that
\begin{equation} \label{kro} \int U_{\lambda} \, d\mu = \int b_{\lambda} - \frac{m(A)}{1-\lambda}  \, d\mu =\frac{1}{1-\lambda}\int (1-\lambda)b_{\lambda} - m(A)  \, d\mu \geq 0.
\end{equation}
In particular this proves that there exists $x_\lambda \in S^1$, such that, $U_{\lambda}(x_\lambda)\geq 0.$

On the other hand, if $w$ is a  subaction we have that
$$A-m(A) \leq w\circ T- w,$$
therefore, using (\ref{io}), for any $\lambda$ and $y$ we have
\begin{align*}
U_\lambda (y) &=\frac{1}{1-\lambda} \, \int A \,d\, \mu_{\lambda}^{y} - \frac{m(A)}{1-\lambda}\leq   \frac{1}{1-\lambda} \left[\int w\circ T\,d\mu_{\lambda}^{y} - \int w\,d\mu_{\lambda}^{y} \right]   \\
&=    \,w(y)-   \int w \,d\,\mu_{\lambda}^{y}.
\end{align*}

Therefore,  the functions $U_{\lambda_{n}}(x)$ are uniformly bounded above.

As the functions $b_{\lambda}$ are equicontinuous in $\lambda<1$ (see \cite{LO}),  the family of functions $U_\lambda$ is equicontinuous. As $U_\lambda$ are uniformly bounded above and $U_\lambda (x_\lambda)\geq 0$ we conclude that this family is also uniformly bounded. 	
\end{proof}

\begin{lemma}\label{calibrated}
Any limit of $U_{\lambda}:=b_{\lambda} - \frac{m(A)}{1-\lambda} $, as $\lambda \to 1$, is a calibrated subaction.	
\end{lemma}

\begin{proof} Let $U$ be the limit of the subsequence $U_{\lambda_n}=b_{\lambda_n} - \frac{m(A)}{1-\lambda_n} $, when $n \to \infty$. From (\ref{blambda}) we get
\[ b_{\lambda}(x) - \frac{m(A)}{1-\lambda} =\sup_{T(y)=x} \lambda \,[ b_\lambda(y) -  \frac{m(A)}{1-\lambda}] + A(y) -  m(A),  \]
that is,
\[U_{\lambda}(x) = \sup_{T(y)=x} \lambda U_{\lambda}(y)  +A(y) -m(A).  \]
Then, as $\lambda_n \to 1$ we conclude that  $U$  is a subaction. Furthermore, for any point $x \in S^1$, there is some point $y_0 \in T^{-1}(x)$ attaining the  supremum of $\sup_{T(y)=x} \lambda_n U_{\lambda_n}(y)  +A(y) -m(A)$, for infinitely many values of $n$. In this way we get
\[U(x)= A(y_0) +U(y_0)-m(A).\]
This proves that $U$ is calibrated.
\end{proof}

\bigskip

\noindent
\textbf{Proof of Theorem \ref{goal}:} We denote $U$ any limit of $U_{\lambda_n}:=b_{\lambda_n} - \frac{m(A)}{1-\lambda_n} $, when $n \to \infty$.
We know that $U$ is a calibrated subaction
and we want to show that $U=V$.

From lemma \ref{dual}, for any maximizing probability $\mu \in \cal{M}$, we have that $\int U\, d\mu \geq 0$. If follows from lemma \ref{lemmai} that $U\geq V$.

\bigskip

Now we will show that $U\leq V$. From lemma \ref{lemmai}  the subaction $V$ satisfies $\int V d \mu \geq 0$, for any $\mu \in \mathcal{M}$, and from  lemma \ref{dual} we get, for any $y$ and $\lambda$, the inequality
$$ U_\lambda (y) \leq  \,V(y)-   \int V \,d\,\mu_{\lambda}^{y}\,.$$
If $\lambda_{n_i}$ is a subsequence of ${\lambda_n}$, such that, $ \mu_{\lambda_{n_i}}^{y}\to \mu_{\infty}$, then, from lemma \ref{se} we have that $\mu_\infty \in \cal{M}$. Therefore, we finally get that
$$U(y)= \lim_{\lambda_{n_i} \to 1} U_{\lambda_{n_i}} (y) \leq  \,V(y)-   \int V \,d\,\mu_\infty \leq V(y).$$
\qed

\begin{corollary}\label{Vcalibrated}
$V$ is a calibrated subaction.
\end{corollary}

\begin{proof}
It is a consequence of lemma \ref{calibrated} and Theorem \ref{goal}.	
\end{proof}

\medskip
\section{Selection for the zero temperature case}

Now we will prove Theorem \ref{goal2}.

We consider for each $\beta>0$ (the inverse of the temperature) and for $\lambda <1$ the operator
\[\mathfrak{S}_{\lambda,\beta}(u)(x) = \log \left(\sum_{T(y)=x}e^{\beta\,A(y) +\lambda u(y)}\right).\]
It is known that $\mathfrak{S}_{\lambda,\beta}$ is a contraction map (see for instance sections 6 and 7 in \cite{Bou}, \cite{LMMS} or \cite{LoT}) with a unique fixed point $u_{\lambda,\beta}$ satisfying
\[e^{u_{\lambda,\beta}(x)} = \sum_{T(y)=x}e^{\beta\,A(y) +\lambda\,u_{\lambda,\beta}(y)}.\]
For each fixed $\beta$, the family $u_{\lambda,\beta}$ is  equicontinuous on $0<\lambda<1$ with uniform constant given by $\beta Lip(A)$. Therefore, for each $\beta$ fixed the function $u_{\lambda,\beta} - \frac{P(\beta\,A)}{1-\lambda}$ is Lipschitz continuous with Lipschitz constant $\beta Lip(A)$.

\begin{lemma}
\[\inf u_{\lambda,\beta} \leq \frac{P(\beta\,A)}{1-\lambda} \leq \sup u_{\lambda,\beta}.\]
\end{lemma}

\begin{proof}
By definition
\[e^{u_{\lambda,\beta}(x)} = \sum_{T(y)=x}e^{\beta\,A(y) +\lambda u_{\lambda,\beta}(y)}.\]
Then, it follows that
\[\sum_{T(y)=x}e^{\beta\,A(y) +\lambda \,u_{\lambda,\beta}(y) - u_{\lambda,\beta}(x)}=1.\]
Therefore,
\[P(\beta\,A +\lambda\,u_{\lambda,\beta} - u_{\lambda,\beta}\circ T)=0.\]
Let $\mu_0$ be the equilibrium probability for $\beta\,A +\lambda\,u_{\lambda,\beta} - u_{\lambda,\beta}\circ T$.

Then,
\[(1-\lambda) \int u_{\lambda,\beta} \, d\mu_{0} = (1-\lambda) \int u_{\lambda,\beta} \, d\mu_{0} + \int (\beta\,A +\lambda\,u_{\lambda,\beta} - u_{\lambda,\beta}\circ T) \, d\mu_{0} + h(\mu_0)\]
\[= \int \beta\,A \, d\mu_{0} + h(\mu_0) \leq P(\beta\,A).\]
It follows that
\[\inf u_{\lambda,\beta} \leq \frac{P(\beta\,A)}{1-\lambda}.\]
On the other hand, if $\mu_1$ is the equilibrium probability of $\beta\,A$, then
\[(1-\lambda) \int u_{\lambda,\beta} \, d\mu_{1} = (1-\lambda) \int u_{\lambda,\beta} \, d\mu_{1} + P(\beta\,A +\lambda\,u_{\lambda,\beta} - u_{\lambda,\beta}\circ T) \]
\[\geq  (1-\lambda) \int u_{\lambda,\beta} \, d\mu_{1} + \int (\beta\,A +\lambda\,u_{\lambda,\beta} - u_{\lambda,\beta}\circ T) \, d\mu_{1} + h(\mu_1)\]
\[=\int \beta\,A \, d\mu_{1} + h(\mu_1) = P(\beta\,A).\]
Therefore,
\[\sup u_{\lambda,\beta} \geq \frac{P(\beta\,A)}{1-\lambda}.\]
\end{proof}

%The Lemma above describes the version of the relations (\ref{re1}), (\ref{re2}), (\ref{re3}) and (\ref{re4}), and its consequences, in the present setting.
% Now we have to consider $\beta$ (positive temperature), and, therefore we need to integrate under equilibrium measures. In the case of zero temperature we have to integrate over maximizing measures.

\begin{lemma}\label{caselambda}
For each fixed $\beta$, the functions $u^{*}_{\lambda,\beta}:=u_{\lambda,\beta}-\frac{P(\beta\,A)}{1-\lambda}$ are Lipschitz functions, with the same Lipschitz constant $H=\beta\,Lip(A)$, and, moreover,  uniformly bounded by $-H$ and $H$. They also satisfy
\[e^{u^{*}_{\lambda,\beta}(x)} = e^{-P(\beta\,A)}\sum_{T(y)=x}e^{\beta\,A(y) +\lambda\,u^{*}_{\lambda,\beta}(y)}.\]
When $\lambda \to 1$, any accumulation function of $e^{u^{*}_{\lambda,\beta}}$ will be an eigenfunction of the Ruelle Operator  $L_{\beta\,A}$ associated to the maximal eigenvalue $e^{P(\beta\,A)}$.
\end{lemma}

\begin{proof} As $u_{\lambda,\beta}$ is equicontinuous the same is true for $u^{*}_{\lambda,\beta}$. Using the equicontinuity (with constant $H$) we have that for any $x$:
\[-H \leq u_{\lambda,\beta}(x)-\sup u_{\lambda,\beta}\leq u_{\lambda,\beta}(x)-\frac{P(\beta\,A)}{1-\lambda} \leq  u_{\lambda,\beta}(x)-\inf u_{\lambda,\beta} \leq H.\]
Furthermore,
\[e^{u^{*}_{\lambda,\beta}(x)} = e^{u_{\lambda,\beta}(x)-\frac{P(\beta\,A)}{1-\lambda}}\]
\[=e^{u_{\lambda,\beta}(x)-\lambda \frac{P(\beta\,A)}{1-\lambda}- P(\beta\,A)}\]
\[=e^{-P(\beta\,A)}\sum_{T(y)=x}e^{\beta\,A(y) +\lambda\,\left(u_{\lambda,\beta}(y)-\frac{P(\beta\,A)}{1-\lambda}\right)}\]
\[=e^{-P(\beta\,A)}\sum_{T(y)=x}e^{\beta\,A(y) +\lambda\,\left(u^{*}_{\lambda,\beta}(y)\right)}.\]
If $u_{\beta}$ is an accumulation function of the family $u^{*}_{\lambda,\beta}$ (when, $\lambda \to 1$), then, we have:
\[e^{u_{\beta}(x)}= e^{-P(\beta\,A)}\sum_{T(y)=x}e^{\beta\,A(y) +u_{\beta}(y)}.\]
\end{proof}

\medskip

{\bf Remark 2:} It is known (Proposition 29 in \cite{CLT}) that
$$ \lim_{\beta \to \infty} (\, P(\beta A) - \beta m(A)\,)= \max_{\mu \in \mathcal{M}} h (\mu).$$

Therefore,
$$ \lim_{\beta \to \infty} \frac{\, P(\beta A)}{ \beta}= m(A),$$
and moreover
$$  \frac{\, P(\beta A)}{ \beta\, (1 - \lambda)}- \frac{ m(A)}{1 - \lambda } \to 0, $$
when $\lambda \to 1$, $\beta \to \infty$ and $\beta (1-\lambda) \to \infty.$

This Remark will be used on the proof of Lemmas 11 and 12.

\medskip

\begin{lemma}
Consider a fixed $\lambda$. Then, when $\beta\to\infty$, the unique possible accumulation point of the family $\frac{1}{\beta}u_{\lambda,\beta}$ is the function $b_{\lambda}$ defined in (\ref{blambda}). Moreover, we get that the unique accumulation point of $\frac{1}{\beta}u^{*}_{\lambda,\beta}$ is the function $b_{\lambda}-\frac{m(A)}{1-\lambda}$.
\end{lemma}

\begin{proof}
As $u^{*}_{\lambda,\beta}$ has Lipschitz constant $\beta\,Lip(A)$ and is bounded by $- \beta\,Lip(A)$ and $\beta\,Lip(A)$, the family  $\frac{1}{\beta}u^{*}_{\lambda,\beta}$ is  equicontinuous  and uniformly bounded by $Lip(A)$. From the limit $\frac{P(\beta\,A)}{\beta(1-\lambda)} \to \frac{m(A)}{1-\lambda}$ (as $\beta\to\infty$) we conclude that (for fixed $\lambda$) the family $\frac{1}{\beta}u_{\lambda,\beta}$ is equicontinuous (with a constant $Lip(A)$) and uniformly bounded.
As
\[\frac{1}{\beta}u_{\lambda,\beta}(x) = \frac{1}{\beta}\log\left(e^{u_{\lambda,\beta}(x)}\right)=\frac{1}{\beta}\log\left(\sum_{T(y)=x}e^{\beta\,A(y) +\lambda\,u_{\lambda,\beta}(y)}\right),\]
there is a unique accumulation point $b$ of $\frac{1}{\beta}u_{\lambda,\beta}$ which satisfies
\[b(x) = \sup_{T(y)=x} [A(y) + \lambda\,b(y)],\]
that is $b=b_{\lambda}$.
\end{proof}

In the previous  section we study the limit of $b_{\lambda}-\frac{m(A)}{1-\lambda}$. Now, we are  interested in  the limit of $\frac{1}{\beta}(u_{\lambda,\beta}-\frac{P(\beta\,A)}{1-\lambda})$, when $\beta\to\infty$ and $\lambda \to 1$.

\begin{lemma} When $\beta \to \infty$ and $\lambda \to 1$, $\frac{(1-\lambda)u_{\lambda,\beta}}{\beta}$ converges uniformly to $m(A)$. If $U$ is a limit of some subsequence of the family $\frac{1}{\beta}(u_{\lambda,\beta}-\frac{P(\beta\,A)}{1-\lambda})$, as $\beta \to \infty$ and $\lambda \to 1$, then $U$ is a calibrated subaction.
\end{lemma}

\begin{proof}
As
\[\left|\frac{1}{\beta}(u_{\lambda,\beta}-\frac{P(\beta\,A)}{1-\lambda})\right| \leq Lip(A)\]
we get
\[\left|\frac{(1-\lambda)u_{\lambda,\beta}}{\beta}-\frac{P(\beta\,A)}{\beta}\right| \leq (1-\lambda)Lip(A).\]
As $\beta\to \infty$ and $\lambda \to 1$, we obtain
\[\frac{(1-\lambda)u_{\lambda,\beta}}{\beta} \to m(A)\]
uniformly, proving the first claim.

\bigskip

In order to prove the second claim we fix a point $x\in S^1$. Let $$U= \lim_{\lambda_{n} \to 1, \, \beta_n \to \infty }\frac{1}{\beta_n}(u_{\lambda_n,\beta_n}-\frac{P(\beta_n\,A)}{1-\lambda_n}).$$ As
\[\sum_{b_0}e^{\beta\,A(\tau_{b_0}x) + \lambda u_{\lambda,\beta}(\tau_{b_0}x) - u_{\lambda,\beta}(x)}=1,\]
we get, for any $b_0 \in\{1,...,d\}$,
\[0\leq u_{\lambda_n,\beta_n}(x) - \lambda_n u_{\lambda_n,\beta_n}(\tau_{b_0}x) - \beta_n A(\tau_{b_0}x)\]
\[= (u_{\lambda_n,\beta_n}-\frac{P(\beta_n A)}{1-\lambda_n})(x) -(u_{\lambda_n,\beta_n}-\frac{P(\beta_n A)}{1-\lambda_n})(\tau_{b_0}x) +\]
\[(1-\lambda_n)u_{\lambda_n,\beta_n}(\tau_{b_0}x) - \beta_n A(\tau_{b_0}x).\]
Dividing the right side by $\beta_n$, taking $\beta_n \to \infty$ and $\lambda_n \to 1$, we get
\[0 \leq U(x) - U(\tau_{b_0}x) + m(A) -A(\tau_{b_0}x).\]
This shows that $U$ is a subaction.

In order to show that $U$ is calibrated, we fix for each $\lambda$ and $\beta$ a point $a=a_{\lambda,\beta}$ maximizing $\beta  A(\tau_ax) + \lambda u_{\lambda,\beta}(\tau_a x) - u_{\lambda,\beta}(x)$. As
\[\sum_{b_0}e^{\beta  A(\tau_{b_0}x) + \lambda u_{\lambda,\beta}(\tau_{b_0}x) - u_{\lambda,\beta}(x)}=1,\]
we conclude that
\[0\leq u_{\lambda_n,\beta_n}(x) - \lambda_n u_{\lambda_n,\beta_n}(\tau_a x) - \beta_n A(\tau_ax) \leq \log(d).\]
When $\beta_n \to\infty$ and $\lambda_n \to 1$, some $a = a_{\lambda_n,\beta_n}$ will be chosen infinitely many times. When $\beta_n \to\infty$ and $\lambda_n \to 1$, this $a$ will  satisfy
\[U(x) - U(\tau_a x) + m(A) -A(\tau_a x) = 0.\]
\end{proof}

In the last section we  proved that the function $$V(x) = \max_{\mu \in   \mathcal{ M}}     \int S(y,x) \, d \mu(y)   $$
 is the unique limit  of the family $b_{\lambda}-\frac{m(A)}{1-\lambda}$. In the present setting, in order to get a similar result, we will assume  a certain condition: $\beta\to\infty$ faster than $\lambda\to 1$, in the sense that $\beta(1-\lambda) \to \infty$.

First we need a Lemma.

\begin{lemma}
\[b_{\lambda}(x)  \leq \frac{1}{\beta}u_{\lambda,\beta} \leq b_{\lambda}(x)  + \frac{\log(d)}{\beta(1-\lambda)}.\]
\end{lemma}

\begin{proof}
From Lemma \ref{caselambda} the function $u^{*}_{\lambda,\beta}=	u_{\lambda,\beta}-\frac{P(\beta A)}{1-\lambda}$ satisfies
	\[e^{u^{*}_{\lambda,\beta}(x)} = e^{-P(\beta\,A)}\sum_{T(y)=x}e^{\beta\,A(y) +\lambda\,u^{*}_{\lambda,\beta}(y)}.\]
Then, for any $a=(a_0a_1a_2...) \in\{1,...,d\}^{\mathbb{N}}$, 	
\begin{align*}
\frac{1}{\beta}{u^{*}_{\lambda,\beta}(x)} &\geq  \,A(\tau_{a_0} x) -\frac{1}{\beta}P(\beta\,A) +\lambda\,\frac{1}{\beta}u^{*}_{\lambda,\beta}(\tau_{a_0}x)\\
&\geq \,A(\tau_{a_0}x)+\lambda A(\tau_{a_1}\tau_{a_0}(x)) -\frac{1+\lambda}{\beta}P(\beta\,A) +\lambda^2\,\frac{1}{\beta}u^{*}_{\lambda,\beta}(\tau_{a_1}\tau_{a_0}x) .
\end{align*}

By induction, as $u^{*}_{\lambda,\beta}$ is uniformly bounded and $\lambda<1$, we obtain
	\[\frac{1}{\beta}{u^{*}_{\lambda,\beta}(x)} \geq S_{\lambda,A}(x,a)-\frac{P(\beta A)}{(1-\lambda)\beta}.\]
	Taking the supremum in $a$ we get
	\[\frac{1}{\beta}{u^{*}_{\lambda,\beta}(x)}\geq b_{\lambda}(x) - \frac{P(\beta A)}{(1-\lambda)\beta},\]
	that is, $\frac{1}{\beta}u_{\lambda,\beta} \geq b_{\lambda}(x)$.
	
On the other hand
\begin{align*}
\frac{1}{\beta}{u^{*}_{\lambda,\beta}(x)} &\leq \frac{1}{\beta}\log(d) +\sup_{a_0}[\,A(\tau_{a_0}x) -\frac{1}{\beta}P(\beta\,A) +\lambda\,\frac{1}{\beta}u^{*}_{\lambda,\beta}(\tau_{a_0}x)]\\
&\leq \frac{1+\lambda}{\beta}\log(d) +\sup_{a_0,a_1}[\,A(\tau_{a_0}x)+\lambda A(\tau_{a_1}\tau_{a_0}(x))\\ &\hspace{1cm} -\frac{1+\lambda}{\beta}P(\beta\,A) +\lambda^2\,\frac{1}{\beta}u^{*}_{\lambda,\beta}(\tau_{a_1}\tau_{a_0}x)].
\end{align*}
Now, we get
\begin{align*}
\frac{1}{\beta}{u^{*}_{\lambda,\beta}(x)} &\leq \frac{\log(d)}{(1-\lambda)\beta} + \sup_{a} S_{\lambda,A}(x,a)-\frac{P(\beta A)}{(1-\lambda)\beta}\\
&= \frac{\log(d)}{(1-\lambda)\beta} +b_{\lambda}(x) - \frac{P(\beta A)}{(1-\lambda)\beta},
\end{align*}
that is, $\frac{1}{\beta}u_{\lambda,\beta} \leq b_{\lambda}(x)  + \frac{\log(d)}{\beta(1-\lambda)}$.
\end{proof}

\noindent
\textbf{Proof of Theorem \ref{goal2}:}
It follows from the above lemma that
\begin{align*} b_{\lambda}(x) - \frac{m(A)}{1-\lambda} &\leq    \frac{1}{\beta}(u_{\lambda,\beta} -\frac{P(\beta A)}{1-\lambda}) + ( \frac{P(\beta A)}{\beta(1-\lambda)}- \frac{m(A)}{1-\lambda} ) \\
&\leq b_{\lambda}(x) - \frac{m(A)}{1-\lambda} + \frac{\log(d)}{\beta(1-\lambda)}.
\end{align*}
Assuming that $\lambda\to 1$, $\beta \to \infty $, $\beta(1-\lambda)\to\infty$, and applying Theorem \ref{goal} we obtain
that $\frac{1}{\beta}(u_{\lambda,\beta} -\frac{P(\beta A)}{1-\lambda}) + ( \frac{P(\beta A)}{\beta(1-\lambda)}- \frac{m(A)}{1-\lambda} )$ converges uniformly to $V$.  As $P(\beta A)= \beta m(A) + \epsilon_\beta$, where $\epsilon_\beta\geq 0$ decreases (see \cite{Conze-Guivarch}), we get  that $\frac{P(\beta A)}{\beta(1-\lambda)}- \frac{m(A)}{1-\lambda}= \frac{\epsilon_\beta}{\beta(1-\lambda)}$ converges to zero. This concludes the proof.
\qed

	\bigskip
	
We finish this section introducing an example (on the symbolic space) where it is studied the limit of $\frac{1}{\beta}\log(\phi_{\beta})$ in a particular case. This limit is not $V$ and this shows that some joint control of $\beta$ and $\lambda$ is really necessary.

\begin{example}

We consider $X=\{0,1\}^{\mathbb{N}}$ with the shift map and a potential $A$ depending on two coordinates. More precisely we suppose $A(1,1)=A(2,2)=0$, $A(1,2)=-5$ and $A(2,1)=-3$.

%\begin{theorem} If $A(1,1)=A(2,2)=0$ then: \newline

%\[(V(1),V(2)) = \left\{\begin{array}{ll}
%\big(\frac{(-3)-A(1,2)}{2},0\big) & if\,(-3)>A(1,2) \\ \\
%\big(0, \frac{A(1,2)-(-3)}{2}\big) & if\, A(1,2)\geq (-3)\end{array}\right.\]

%\end{theorem}

Consider the matrix $$L_{\beta}=\left(\begin{array}{cc} e^{\beta A(1,1)} & e^{\beta A(1,2)}\\ e^{\beta A(2,1)} & e^{\beta A(2,2)}\end{array}\right) =  \left(\begin{array}{cc} 1 & e^{-5\beta}\\ e^{-3\beta } & 1\end{array}\right)$$
that defines the Ruelle Operator associated to $\beta A$. We note that the main eigenvalue is given by
\[\alpha_{\beta}=e^{P(\beta A)}=1+e^{-4\beta}.\]
Furthermore, the eigenfunction $\phi_\beta$ associated to the Ruelle operator of $\beta A$ depends on the first coordinate and satisfies
\[\phi_\beta(1)=1+e^{\beta}\,\,\, and\,\,\, \phi_\beta(2)=1+e^{-\beta}\]
(it  can be directly checked that $\phi_\beta L_\beta= \alpha_\beta \phi_\beta$ ). Any multiple of $\phi_\beta$ is also an eigenfunction.
When $\beta\to +\infty$, we get
$U= \lim_{\beta\to\infty}\frac{1}{\beta}\log(\phi_\beta)$ which satisfies
\[U(1) = 1, \,\,\, U(2) = 0.\]
Now we will prove that $V\neq U$. Indeed
any maximizing measure for $A$ is of the form $\mu_r:=r\delta_{1^{\infty}} + (1-r)\delta_{2^\infty}$.
Consider the functions
\[S(y,x) = \lim_{\varepsilon \to 0} \sup_n\{ S_n(A)(y'): d(y',y)<\varepsilon, \, T^{n}(y')=x\}\]
and
\[V(x) = \sup_{\mu_r}\int S(y,x) \, d\mu_r(y).\]
As $A$ depends only on two coordinates we conclude that $S(y,x)$ depends only on the first coordinate of $x$. From the analysis of $S(y,x)$ we get
\[S(2^{\infty}, 1) = -3,\,\, S(1^{\infty},1) = 0,\,\, S(1^{\infty},2) = -5,\,\, S(2^{\infty},2) = 0 \]
(for instance, when considered $S(2^{\infty}, 1)$ the point $y'$ in (\ref{forg}) will contain  the word $21$, and each word $21$ will decrease the value of $S_n(A)(y')).$

Then,
\[\sup_{\mu_r} \int S(y,1) \, d\mu_r(y) = \sup_r [r\cdot 0 +(1-r)(-3)]  = 0\]
and
\[\sup_{\mu_r} \int S(y,2) \, d\mu_r(y) = \sup_r [r (-5)+(1-r)(0)]  = 0.\]
This shows that $V \neq U.$
\end{example}

\bigskip

Renato Iturriaga

renato@cimat.mx

CIMAT- Guanajuato - Mexico

\medskip

Artur O. Lopes

arturoscar.lopes@gmail.com

IME - UFRGS - Porto Alegre - Brasil

\medskip

Jairo K. Mengue

jairokras@gmail.com

IME - UFRGS - Porto Alegre - Brasil

\end{document}